\documentclass[preprint,12pt]{elsarticle}

\usepackage{amssymb}
\usepackage{amsthm}
\usepackage{amsmath}
\usepackage{booktabs}
\usepackage{hyperref}
\usepackage{verbatim}
\usepackage{mathtools}
\usepackage[utf8]{inputenc}
\usepackage[english]{babel}
\usepackage{todonotes}
\usepackage{bm}
\usepackage{fullpage}
\usepackage{algorithm}
\usepackage{algpseudocode}

\usepackage{tikz}
\usepackage{pgfplots,pgfplotstable}
\usepgfplotslibrary{colorbrewer,groupplots}
\usepackage{caption}
\usepackage{subcaption}

\usepackage{pgfplotstable}
\usepackage{tabu}
\usepackage{longtable}
\usepackage{booktabs}
\pgfplotsset{compat=1.18}

\pgfplotsset{
cycle list/Set1-5,
cycle multiindex* list={
mark list*\nextlist
Set1-5\nextlist
},
}

\pgfplotsset{
    discard if not/.style 2 args={
        x filter/.append code={
            \edef\tempa{\thisrow{#1}}
            \edef\tempb{#2}
            \ifx\tempa\tempb
            \else
                
            \fi
        }
    }
}

\newtheorem{theorem}{Theorem}
\newtheorem{lemma}[theorem]{Lemma}
\newtheorem{proposition}{Proposition}
\newtheorem{corollary}[theorem]{Corollary}
\newtheorem{definition}[theorem]{Definition}
\newtheorem{remark}[theorem]{Remark}
\newtheorem{ex}[theorem]{Example}

\newcommand{\ei}{e^{(i)}}
\newcommand{\lami}{\lambda^{(i)}}
\newcommand{\ui}{u^{(i)}}
\newcommand{\norm}[1]{\lVert #1 \rVert}
\DeclareMathOperator{\Id}{Id}

\journal{Computer and Mathematics with Applications}

\begin{document}

\begin{frontmatter}



\title{An adaptive mesh refinement strategy to ensure quasi-optimality of finite element methods for self-adjoint Helmholtz problems}


\author{Tim van Beeck$\,^\text{a}$}

\affiliation{organization={Institute for Numerical and Applied Mathematics, University of Göttingen},
            addressline={Lotzestraße 16-18},
            city={Göttingen},
            postcode={37073},
            country={Germany}}
            
\author{Umberto Zerbinati$\,^\text{b}$}

\affiliation{organization={Mathematical Institute, University of Oxford},
            addressline={Andrew Wiles Building, Radcliffe Observatory Quarter, Woodstock Road},
            city={Oxford},
            postcode={OX2 6GG},
            country={United Kingdom}}

\begin{abstract}
It is well known that the quasi-optimality of the Galerkin finite element method for the Helmholtz equation is dependent on the mesh size and the wave-number.
In the literature, different criteria have been proposed to ensure uniform quasi-optimality of the discretisation.
In the present work, we study the uniform quasi-optimality of $H^1$ conforming and non-conforming Crouzeix-Raviart discretisation of the self-adjoint Helmholtz problem.
In particular, we propose an adaptive scheme, coupled with a residual-based indicator, for generating guaranteed quasi-optimal meshes with minimal degrees of freedom.
\end{abstract}

\end{frontmatter}


\section{Introduction}

In this article, we consider the self-adjoint Helmholtz equation: find $u : \Omega \to \mathbb{R}$ such that 
\begin{equation}\label{eq:helmholtz}
    - \Delta u - k^2 u = f \quad \text{in } \Omega, \quad u = 0 \quad \text{on } \Gamma_D, \quad \bm{n} \cdot \nabla u = 0 \quad \text{on } \Gamma_N,
\end{equation}
where we assume that $\Omega \subset \mathbb{R}^d$, $d = 2,3$, is a bounded Lipschitz domain with boundary $\partial \Omega = \Gamma_D \cup \Gamma_N$, where $\Gamma_D$ and $\Gamma_N$ are such that $\Gamma_D \cap \Gamma_N = \emptyset$ and $\bm{n}$ is the outward unit normal. Our interest in the Helmholtz equation arises from its wide range of applications in acoustics, electromagnetics, elastodynamics, and seismology.

In particular, we are interested in Galerkin discretisations of the self-adjoint Helmholtz equation, where we seek to approximate the variational formulation resulting from \eqref{eq:helmholtz} in a finite dimensional space $V_h$, usually consisting of piecewise polynomials of order $p$. A well-known issue in the analysis of such discretisations is that the resulting bilinear form lacks coercivity for large $k^2$.
This leads to the so-called \emph{pollution effect}, i.e.~for a fixed mesh, we lose quasi-optimality of the Galerkin method as the wave number $k^2$ increases. Ensuring that the discretisation is quasi-optimal requires, for fixed polynomial degree $p$, that the mesh size $h$ is sufficiently small. In particular, this requirement is typically more restrictive than the requirement $\frac{kh}{2\pi p} \le 1$, which is needed to accurately resolve the wave. 

We are here interested in the so-called \textit{asymptotic regime}, i.e.~where uniform quasi-optimality is guaranteed, see \cite{AKS88,BS00,BCFM24,N23,N19}.
Another area of active study is concerned with the \textit{preasymptotic} regime, where on milder assumptions on the mesh size the relative error of the Galerkin solution is controllably small, see e.g.~\cite{Wu14,DW15,GS24}.

Adpative finite element methods have shown to be extremely effective in the context of the Helmholtz equation \cite{BHP17, HS13, CFEV21, DS13, SZ15}.
For example in \cite{MS10, MS11,BCFM24} it has been proven that the hp-adaptive finite element method can achieve quasi-optimality for the Helmholtz equation under the following assumptions: 
\begin{equation}
    \label{eq:hp}
    \frac{k h}{p} \ll 1, \text{ and } p \geq C_{hp}\log(k).
\end{equation}

In the present work, we aim to study the discretisation of the self-adjoint Helmholtz equation \eqref{eq:helmholtz} with $H^1$-conforming and non-conforming Crouzeix-Raviart finite elements.
We present an analysis based on the notion T-coercivity \cite{BCZ10,Cia12,BCC14,BCC18}, which highlights the relationship between the stability of the discrete problem and a sufficiently accurate approximation of the eigenvalues of the Laplace operator.
Based on this analysis, we propose a practical algorithm that produces a mesh which guarantees the quasi-optimality of the discretisation. 
Until the approximation of the Laplace eigenvalues is sufficiently accurate, we refine the mesh either uniformly or adaptively.
In the latter case we use an adaptive error estimator for the eigenvalue problem \cite{BR78}.
Similar ideas have been discussed also in \cite{D23}, but to the best of the authors' knowledge, they have not been explored in details nor have been implemented to construct quasi-optimal meshes.
For the Crouzeix-Raviart discretisation, we leverage guaranteed lower bounds for the eigenvalue approximations as derived in  \cite{CG14,CP24}.

This paper is organized as follows: in Section \ref{sec:abstractFramework} we introduce the abstract framework, which we apply in Section \ref{sec:continuousProblem} to analyse the continuous problem. In Section \ref{sec:discreteAnalysis}, we present a unified analysis that covers both $H^1$-conforming and non-conforming Crouzeix-Raviart discretisations.
Afterward, in Section \ref{sec:gmr}, we discuss a mesh generation strategy which ensures the quasi-optimality of the discretisations. Throughout Sections \ref{sec:discreteAnalysis} and \ref{sec:gmr}, we present numerical results validating our observations. All numerical experiments presented in this manuscript are implemented in Python using the Firedrake finite element library \cite{Firedrake} and ngsPETSc \cite{ngsPETSc}.

\section{Abstract framework}
\label{sec:abstractFramework}
In this section, we provide a brief overview of the abstract framework we will use to analyze \eqref{eq:helmholtz}. In particular, we introduce the notion of T-coercivity \cite{BCZ10, Cia12} and discuss its application to study the well-posedness of variational problems and their approximation with Galerkin type methods.
In the following, let $V$ be a Hilbert space and $\mathcal{A} : V \times V \rightarrow \mathbb{R}$ be a continuous bilinear form.

\begin{definition}[T-coercivity]
    The bilinear form $\mathcal{A}$ is \textit{T-coercive} on $V$ if there exists a constant $\alpha > 0$ and a bijective operator $T : V \rightarrow V$ such that
    \begin{equation*}
        \mathcal{A} (u,Tu) \ge \alpha \Vert u \Vert^2_{V} \quad \forall u \in V.   
    \end{equation*}
\end{definition}

\noindent For $f \in L^2(\Omega)$, we are concerned with the well-posedness of the problem:
\begin{equation}\label{eq:abstractProblem}
    \text{find } u \in V \text{ such that } \mathcal{A}(u,v) = (f,v)_{L^2(\Omega)} \quad \forall v \in V. 
\end{equation}
The T-coericvity of $\mathcal{A}$ is equivalent to the well-posedness of \eqref{eq:abstractProblem} as shown in \cite[Thm.~1]{Cia12}.

In Section \ref{sec:discreteAnalysis}, we will consider both, an $H^1$-conforming and non-conforming Crouzeix-Raviart discretisations.
Thus, we consider the following setting which is general enough to cover both discretisations, but restrictive enough to avoid the more technical aspects of a fully discontinuous setting.
Let $\{ V_h \}_{h \in \mathcal{H}}$ be a sequence of finite-dimensional spaces, which do not necessarily have to be subspaces of $V$. We assume that the spaces $V_h$ are equipped with a discrete norm $\Vert \cdot \Vert_{V_h}$ which is well-defined on $V + V_h$. 
Furthermore, let $\mathcal{A}_h : V_h \times V_h \to \mathbb{R}$ be a discrete bilinear form which is continuous with respect to $\Vert \cdot \Vert_{V_h}$ on $V + V_h$ and consistent in the sense that $\mathcal{A}_h(u,v_h) = (f,v_h)_{L^2(\Omega)}$ for all $v_h \in V_h$ when $u \in V$ is the solution of \eqref{eq:abstractProblem}.

We consider the discrete problem
\begin{equation}\label{eq:abstractDiscreteProblem}
    \text{find } u_h \in V_h \text{ such that } \mathcal{A}_h(u_h,v_h) = (f,v_h)_{L^2(\Omega)}\quad \forall v_h \in V_h.
\end{equation}

\begin{definition}
    The bilinear form $\mathcal{A}_h$ is called uniformly T$_h$-coercive on $\{ V_h \}_{h \in \mathcal{H}}$, if there exists a family of bijective operators $\{ T_h \}_{h \in \mathcal{H}}$, $T_h : V_h \to V_h$, and a constant $\alpha^{\ast}> 0$ independent of $h$ such that
    \begin{equation}
        \mathcal{A}_h(u_h,T_h u_h) \ge \alpha^{\ast} \Vert u_h \Vert^2_{V_h} \quad \forall u_h \in V_h. 
    \end{equation} 
\end{definition}

It has been shown that $\mathcal{A}_h$ is uniformly T$_h$-coercive on $V_h$ if and only if the bilinear form $\mathcal{A}_h$ fulfills an uniform inf-sup condition on $V_h$ and therefore the stability of the discrete problem is equivalent to the uniform T$_h$-coercivity of $\mathcal{A}_h$: 

\begin{theorem}[Thm.~2 of \cite{Cia12}]\label{thm:Ciarlet2}
    Let $\mathcal{A}_h$ be uniformly T$_h$-coercive on $V_h$ with constant $\alpha^{\ast} > 0$. Then, the discrete problem \eqref{eq:abstractDiscreteProblem} is well-posed and stable. Further, it holds that 
    \begin{equation}
        \Vert u - u_h \Vert_{V_h} \le C_{\text{qo}} \inf_{v_h \in V_h} \Vert u - v_h \Vert_{V_h},
    \end{equation}
    where $C_{\text{qo}} := 1+ (\Vert \mathcal{A}_h \Vert_{\mathcal{L}(V_h,V_h)} \Vert T_h \Vert_{\mathcal{L}(V_h,V_h)})/ \alpha^{\ast} > 0$. 
\end{theorem}

We note that even in the conforming setting, the uniform T$_h$-coercivity of $\mathcal{A}_h$ has to be shown on $V_h$, because unlike coercivity, T-coercivity is not inherited by the discrete problem.
In the conforming setting, it would be sufficient to show that
\begin{equation}
    \label{eq:Tcoercive:ThtoT}
    \lim_{h \rightarrow 0} \Vert T_h - T \Vert_{\mathcal{L}(V,V)} = 0 
\end{equation}
to ensure the uniform T$_h$-coercivity of $\mathcal{A}_h$, cf.~\cite[Cor. 1]{Cia12}. Similar criteria can also be developed in a more general setting, applicable  also to non-conforming discretisations, cf.~\cite{Halla21Tcomp,HLS22} for T-compatibility and \cite{H23,vB23} for its applications to discontinuous Galerkin methods.

In the following, we will avoid the use of these criteria in our analysis to be more explicit on the assumptions on the mesh size $h$ required to ensure the quasi-optimality of the Galerkin method.

\section{Well-posedness of the continuous problem}\label{sec:continuousProblem}
\noindent We want to show the well-posedness of the continuous weak formulation associated with \eqref{eq:helmholtz}. To this end, we define the Hilbert space 
\begin{equation}
    V := \{ u \in H^1(\Omega) : \gamma_0 u = 0 \text{ on } \Gamma_D \},
\end{equation}
where $\gamma_0 : H^1(\Omega) \to L^2(\partial \Omega)$ is the trace operator. In the following, let $a(u,v) := (\nabla u, \nabla v)_{L^2(\Omega)}$, $u,v \in V$, be the bilinear form associated with the weak Laplacian $- \Delta$ on $V$. The weak formulation of problem \eqref{eq:helmholtz} reads as: find $u \in V$ such that
\begin{equation}\label{eq:weakForm}
    \mathcal{A}(u,v) := a(u,v) - k^2 (u,v)_{L^2(\Omega)} = (f,v)_{L^2(\Omega)} \quad \quad \forall v \in V. 
\end{equation}
Let $(\lami, \ei)_{i \in \mathbb{N}}$ be the eigenpairs of the Laplacian with homogeneous Dirichlet boundary conditions, i.e.~the solutions to the problem: find $\lami \in \mathbb{R}$ and $\ei \in V \cap H^1_0(\Omega)$ such that
\begin{equation}
    \label{eq:EVP:cont}
    a(\ei,v) = \lami (\ei,v)_{L^2(\Omega)} \quad \forall v \in V \cap H^1_0(\Omega).
\end{equation}
Let $S : L^2(\Omega) \to L^2(\Omega)$ be the solution operator associated to \eqref{eq:EVP:cont}. Due to the Poincar\'e-inequality, the bilinear form $a(\cdot,\cdot)$ is coercive on $V \cap H^1_0(\Omega)$ and thus the well-posedness of problem \eqref{eq:EVP:cont} is a consequence of the Lax-Milgram lemma.
Hence, the solution operator $S$ is well-defined and due to the compactness of the embedding $H^1(\Omega) \hookrightarrow L^2(\Omega)$ the operator $S$ is also compact.
Additionally, $S$ is self-adjoint and therefore the eigenfunctions $\{ \ei \}_{i \in \mathbb{N}}$ form a Hilbert basis of $L^2(\Omega)$ \cite{B10,Acta10}.

Since $V \subset L^2(\Omega)$, we can represent any function $u \in V$ as $u = \sum_{i \in \mathbb{N}} \ui \ei$ with coefficients $\ui \in \mathbb{R}$.
Without loss of generality, we assume that the eigenfunctions $\{\ei\}_{i \in \mathbb{N}}$ are normalized with respect to the $H^1$-norm and that the eigenvalues $\{\lami\}_{i\in \mathbb{N}}$ are ordered by increasing value and counted with their algebraic multiplicity.

From now on, we assume that the wave-number $k$ is such that $k^2 > \lambda^{(1)}$, otherwise the problem is coercive and the well-posedness of \eqref{eq:weakForm} follows directly from the Lax-Milgram lemma.
Then, we denote by $i_\ast \coloneqq \max \{ i \in \mathbb{N} : \lami < k^2 \}$ the largest index such that $\lambda^{(i_{\ast})}$ is smaller than $k^2$.

We define the subspace $W \subset V$ and the map $T : V \rightarrow V$ as follows:
\begin{equation}
    W \coloneqq \text{span}_{0 \le i \le i_\ast} \{ \ei \} \subset V, \qquad T \coloneqq \Id_V - 2 P_W,
\end{equation}
where $P_W \in \mathcal{L}(V,W)$ is the orthogonal projection onto $W$. Note that $T$ is self-inverse and therefore bijective and acts on the eigenfunctions as
\begin{equation}\label{eq:T}
    T \ei := \begin{cases}
        - \ei & \text{if } 0 \le i \le i_\ast, \\
        + \ei & \text{if } i > i_\ast.
    \end{cases}
\end{equation}

\begin{lemma}
    \label{lem:aD:Tcoercive}
    Assume that $k^2 \not \in \{ \lami \}_{i \in \mathbb{N}}$. Then, the bilinearform $\mathcal{A}$ defined in \eqref{eq:weakForm} is T-coercive on $V$.  
\end{lemma} 
To showcase the mechanism behind the T-coercivity of Helmholtz-like problems, we report the proof of \cite[Prop. 1]{Cia12}.
\begin{proof}
    By definition, $T$ swaps the sign of the eigenfunctions corresponding to eigenvalues smaller than $k^2$. Therefore, we have that 
    \begin{align*}
        \mathcal{A}(u,Tu) &:= \sum_{0 \le i \le i_\ast} \left( \frac{k^2 - \lami}{1 + \lami} \right) (\ui)^2 + \sum_{i > i_\ast} \left(\frac{\lami - k^2}{1 + \lami} \right) (\ui)^2 \\
        &\ge \alpha \sum_{i \in \mathbb{N}} \lami (\ui)^2 \ge \alpha \Vert u \Vert^2_{H^1(\Omega)} \quad \forall u \in V,
    \end{align*}
    where $\alpha := \min_{i \ge 0} \left\{ \left\vert \frac{\lami - k^2}{1 + \lami} \right\vert \right\} > 0$. 
\end{proof}

As previously discussed, the T-coercivity of the bilinear form $\mathcal{A}$ immediately yields the well-posedness of \eqref{eq:weakForm} for $k^2 \not \in \{ \lami \}_{i \in \mathbb{N}}$. 

\section{Discrete analysis}\label{sec:discreteAnalysis}
In this section, we consider the finite element approximation of the problem \eqref{eq:weakForm}.
The main goal is to show that the resulting discretisations are uniformly T$_h$-coercive which allows us to apply Thm.~\ref{thm:Ciarlet2} to conclude the stability of the discrete problems. 

In the following, we assume that $\Omega$ is a polygonal domain and that $\{ \mathcal{T}_h \}_{h \in \mathcal{H}}$ is a sequence of shape regular triangulations of $\Omega$ with mesh size $h$.
Let $\{ V_h \}_{h \in \mathcal{H}} \subset L^2(\Omega)$ be a sequence of finite dimensional spaces, which do not necessarily have to be subspaces of $V$.
Then, the discretisation of \eqref{eq:weakForm} reads as: find $u_h \in V_h$ such that
\begin{equation}\label{eq:discreteProblem}
    \mathcal{A}_h(u_h,v_h) := a_h(u_h,v_h) - k^2 (u_h,v_h)_{L^2(\Omega)} = (f,v_h)_{L^2(\Omega)} \quad \forall v_h \in V_h,
\end{equation}
where $a_h(u_h,v_h) := \sum_{K \in \mathcal{T}_h} (\nabla u_h, \nabla v_h)_{K}$ is the bilinear form associated with the weak Laplacian on the discrete space $V_h$.

We will denote $V_{h,0}$ the subspace of $V_h$ incorporating homogeneous Dirichlet boundary conditions and define  $n_h := \dim V_{h,0}$. Let $(\lambda_h^{(i)},e_h^{(i)})_{0 \le i \le n_h}$ be the solutions of the discrete eigenvalue problem: find $\lami_h \in \mathbb{R}$ and $\ei_h \in V_{h,0}$ such that 
\begin{equation}
    \label{eq:discreteEVP}
    a_h(e_h,v_h) = \lambda_h (e_h,v_h)_{L^2(\Omega)} \qquad \forall v_h \in V_{h,0}.
\end{equation}
In the following, we assume that the discrete eigenvalue problem is well-posed and that the sequence of associated solution operators $\{ S_h \}_{h \in \mathcal{H}}$, $S_h : L^2(\Omega) \to L^2(\Omega)$, converges uniformly in operator norm to the solution operator $S$ of the continuous eigenvalue problem \eqref{eq:EVP:cont}.
This assumption guarantees that the sequence of discrete eigenpairs $(\lami_h, \ei_h)_{0 \le i \le n_h}$, which we assume to be ordered by increasing value of the eigenvalues, approximates the continuous eigenpairs $(\lami,\ei)_{i \in \mathbb{N}}$ as $h \to 0$.
In particular, this rules out the existence of \emph{spurious modes}, which might pollute the discrete spectrum \cite{Acta10}.

As in the continuous setting, we set 
\begin{equation}
    W_h := \text{span}_{0 \le i \le i_{\ast}} \{ \ei_h \} \subset V_h,
\end{equation}
and define a bijective operator $T_h : V_h \to V_h$ through $T_h := \Id_{V_h} - 2P_{W_h}$, where $P_{W_h} \in \mathcal{L}(V_h,W_h)$ is the orthogonal projection onto $W_h$. Note that $T_h$ acts on eigenfunctions as
\begin{equation}
    T_h(\ei_h) = \begin{cases}
        - \ei_h & \text{if } 0 \le i \le i_{\ast}, \\
        + \ei_h & \text{if } i > i_{\ast}.
    \end{cases}
\end{equation}

The following result shows that the bilinear form $\mathcal{A}_h(\cdot,\cdot)$ is uniformly T$_h$-coercive and therefore stable by Thm.~\ref{thm:Ciarlet2} provided that the mesh size $h$ is small enough. 
We will shortly discuss in greater detail how small $h$ has to be.

\begin{theorem}\label{thm:mainresult}
    Under the previous assumptions on the discrete eigenvalue problem \eqref{eq:discreteEVP}, the bilinear form $\mathcal{A}_h(\cdot,\cdot)$ is uniformly $T_h$-coercive if $h$ is such that $\lambda_h^{(i_{\ast})} < k^2 < \lambda_h^{(i_{\ast} + 1)}$.
\end{theorem}

\begin{proof}
    Let $u_h \in V_h$. Using the discrete eigenbasis, we expand $u_h$ as $u_h = \sum_{0 \le i \le n_h} \ui_h \ei_h$.
    In analogy with the continuous case, we can evaluate $\mathcal{A}_h(u_h,T_h u_h)$ as
    \begin{equation}
        \mathcal{A}_h( u_h, T_h u_h) = \sum_{0 \le i \le i_{\ast}} \left( \frac{k^2 - \lami_h}{1 + \lami_h} \right) (\ui_h)^2 + \sum_{i_{\ast} < i \le n_h} \left( \frac{\lami_h - k^2}{1 + \lami_h} \right) (\ui_h)^2.
    \end{equation}
    Provided that $\lambda_h^{(i_{\ast})} < k^2 < \lambda_h^{(i_{\ast} + 1)}$, all terms are positive and $\mathcal{A}_h$ is uniformly $T_h$-coercive with constant $\alpha^\ast := \min_{i \ge 0} \left\{ \left\vert \frac{\lami_h - k^2}{1 + \lami_h} \right\vert \right\}$, see also Fig.~\ref{fig:Thcoercive}.
\end{proof}

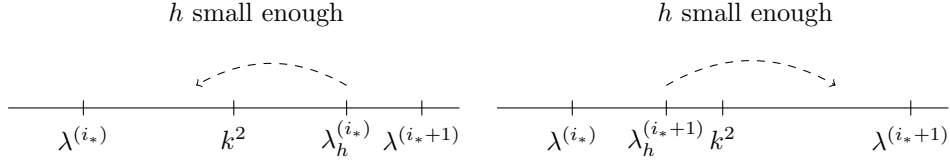
\begin{figure}
    \centering
    \begin{tikzpicture}
        \begin{scope}
            \draw (0,0) -- (6,0);
            \draw (1,-0.1) -- (1,0.1);
            \node at (1,-0.4) {\footnotesize $\lambda^{(i_{\ast})}$};
            \draw (3,-0.1) -- (3,0.1);
            \node at (3,-0.4) {\footnotesize $k^2$};
            \draw (4.5,-0.1) -- (4.5,0.1);
            \node at (4.5,-0.4) {\footnotesize $\lambda^{(i_{\ast})}_h$};
            \draw (5.5,-0.1) -- (5.5,0.1);
            \node at (5.5,-0.4) {\footnotesize $\lambda^{(i_{\ast}+1)}$};
            \draw[->,dashed] (4.5,0.3) to [out=150,in=30] (2.5,0.3);
            \node at (3.3,1.25) {\footnotesize $h$ small enough};
        \end{scope}
        \begin{scope}[xshift=6.5cm]
            \draw (0,0) -- (6,0);
            \draw (1,-0.1) -- (1,0.1);
            \node at (1,-0.4) {\footnotesize $\lambda^{(i_{\ast})}$};
            \draw (3,-0.1) -- (3,0.1);
            \node at (3,-0.4) {\footnotesize $k^2$};
            \draw (5.5,-0.1) -- (5.5,0.1);
            \node at (5.5,-0.4) {\footnotesize $\lambda^{(i_{\ast}+1)}$};
            \node at (3.3,1.25) {\footnotesize $h$ small enough};
            \draw[] (2.25,-0.1) -- (2.25,0.1); 
            \node at (2.25,-0.4) { \footnotesize $\lambda^{(i_{\ast}+1)}_h$};
            \draw[->,dashed] (2.25,0.3) to [out=30,in=150] (4.5,0.3);
        \end{scope}
    \end{tikzpicture}
    \caption{We assume the discrete eigenvalues to be ordered but, depending on the discretisation, it might be the case that $\lambda_h^{(i_{\ast})} > k^2$ (left) or $\lambda^{(i_{\ast}+1)} < k^2$ (right). We achieve T$_h$-coercivity by choosing $h$ small enough such that $\lambda^{(i_{\ast})}_h < k^2 $ and $\lambda^{(i_{\ast}+1)}_h > k^2$, respectively.}
    \label{fig:Thcoercive}
\end{figure}

Once we have shown that the discrete bilinear form $\mathcal{A}_h(\cdot,\cdot)$ is uniformly $T_h$-coercive on $V_h$, we can use classical arguments to prove that the approximation error is bounded by the best approximation error in $V_h$, cf.~Thm.~\ref{thm:Ciarlet2}.
It can be shown, with further arguments, that the constants are independent of the wave number $k^2$ \cite{MS10,MS11,CF24}.

\subsection{$H^1$-conforming finite element discretisation}
For $p \ge 1$ we define $V_h \subset V$ to be the $H^1$-conforming finite element space of order $p$, i.e. 
\begin{equation*}
    V_h := \{ v \in H^1(\Omega) : v \vert_T \in \mathcal{P}^p(T) \quad \forall T \in \mathcal{T}_h \} \cap V.  
\end{equation*}
We have that $a_h(\cdot,\cdot) = a(\cdot,\cdot)$ and therefore $\mathcal{A}_h(\cdot,\cdot) = \mathcal{A}(\cdot,\cdot)$ on $V_h$. Since $V_h \subset V$ is a conforming subspace, the coercivity of $a(\cdot,\cdot)$ is inherited uniformly to the discrete level, hence problem \eqref{eq:discreteEVP} is well-posed.
Furthermore, it is well-known that the associated sequence of discrete solution operators converges to the solution operator of the continuous problem \cite{Acta10}. Thus, the assumptions on the discrete eigenvalue problem \eqref{eq:discreteEVP} are satisfied and Thm.~\ref{thm:mainresult} applies.

The following result shows that we can leverage eigenvalue estimates to derive an upper bound on the mesh size $h$ such that the condition $\lambda_h^{(i_{\ast})} < k^2$ is satisfied.

\begin{corollary}
    Assume that $h^{2q} < C \left( \frac{k^2 - \lambda^{(i_\ast)}}{4 \sqrt{i_{\ast}} \lambda^{(i_\ast)}} \right)$, where $q \coloneqq \min \{p,s-1\}$ with $s$ being the Sobolev index granted by elliptic regularity for an $H^1(\Omega)$-data. Then $\lambda_h^{(i_{\ast})}  < k^2 < \lambda_h^{(i_{\ast}+1)}$. 
\end{corollary}

\begin{proof}
    Since conforming discretisations approximate eigenvalues from above \cite{Acta10}, we only have to consider the index $i_{\ast}$, and it suffices to ensure that 
    \begin{align*}
        \lambda_h^{(i_{\ast})} - \lambda^{(i_{\ast})} < k^2 - \lambda^{(i_{\ast})}. 
    \end{align*}
    To bound the left-hand side, we can use approximation results for the eigenvalues. For instance \cite[Thm. 47.10]{EG_FE2} yields 
    \begin{align*}
        \lambda_h^{(i_{\ast})} - \lambda^{(i_{\ast})} &\le \lambda^{(i_\ast)} 4 \sqrt{i_{\ast}} C \max_{v \in S_{i_{\ast}}} \min_{v_h \in V_h} \Vert v - v_h \Vert^2_V \\
        &\le C \lambda^{(i_\ast)} 4 \sqrt{i_{\ast}} h^{2q},
    \end{align*}
    where $S_{i_{\ast}}$ is the unit sphere, with resepct to $\norm{\cdot}_{L^2(\Omega)}$, in the span of the first $i_{\ast}$-eigenfunctions. Thus, we calculate 
    \begin{align*}
        C \lambda^{(i_\ast)} 4 \sqrt{i_{\ast}} h^{2q} < k^2 - \lambda^{(i_{\ast})} 
        \Longleftrightarrow \quad h^{2q} < \frac{1}{C} \left( \frac{k^2 - \lambda^{(i_\ast)}}{4 \sqrt{i_{\ast}} \lambda^{(i_\ast)}} \right).
    \end{align*}
\end{proof}
The constant $C$ in the previous corollary is in general not computable, because it involves a constant related to the elliptic regularity shift of the Laplace operator.
\begin{remark}
    From the previous corollary, we see that as the wave number $k$ gets closer to the eigenvalue $\lambda^{(i_{\ast})}$, the mesh size $h$ has to be chosen smaller to ensure the quasi-optimality of the discretisation.
    Similiar observations can be found in \cite{CFV22,CF24}.
\end{remark}

\begin{ex}[Unit-square]
    As a first numerical example, we consider the case of pure Dirichlet boundary conditions, i.e.~$\Gamma_D = \partial \Omega$, on the unit-square $\Omega = [0,1]^2 \subset \mathbb{R}^2$. On this geometry, the eigenvalues of the Laplacian are given by 
    \begin{equation}\label{eq:unitSquare:EV}
        \lambda_{i,j} = \pi^2 (i^2 + j^2), \quad i,j \in \mathbb{N}.
    \end{equation}
    Since we know the exact eigenvalues, we can determine the index $i_{\ast}$ directly. We solve the discrete problem \eqref{eq:discreteProblem} with $H^1$-conforming elements of order $p = 1$ and $p = 2$. Figure \ref{fig:unitsquareDirichlet} shows that the condition $\lambda_h^{(i_{\ast})} < k^2$ is indeed sufficient to ensure quasi-optimality.
    \begin{figure}[!htbp]
        \begin{center}
            \begin{tikzpicture}[scale=0.75]
                \begin{groupplot}[%
                    group style={%
                    group size=2 by 2,
                    horizontal sep=1.5cm,
                    vertical sep=2cm,
                    },
                ymajorgrids=true,
                grid style=dashed,
                ]       
                \nextgroupplot[width=9cm,height=7cm,domain=2:5,xmode=log,ymode=log, xlabel={$h$}, ylabel={$\norm{u-u_h}_{L^2(\Omega)}$}, title={$k^2=100$}, 
                legend style={legend columns=3, draw=none,nodes={scale=.8}}, 
                legend to name=named]
        
                \addplot+[orange,line width=1.5pt,mark=None] table [x=h, y=error, col sep=comma] {error_omega10.0_p1_square.csv};
                \addplot+[violet,line width=1.5pt,mark=None] table [x=h, y=error, col sep=comma] {error_omega10.0_p2_square.csv}; 
                
                \addplot[gray,dashed,domain=1e-2:0.06] {350271*x^2};
                \node [draw=none] at (axis description cs:0.25,0.7) {\color{gray}\footnotesize $\!\!\mathcal{O}(h^{2})$};
                
                \addplot[gray,dashed,domain=1e-2:0.2] {25000*x^3};
                \node [draw=none] at (axis description cs:0.15,0.25) {\color{gray}\footnotesize $\!\!\mathcal{O}(h^{3})$};

                \draw[orange, dashed,thick] (0.0419,1e-6) -- (0.0419,1e6);
                \draw[violet, dashed,thick] (0.1638,1e-6) -- (0.1638,1e6); 
                \legend{$p=1\quad$,$p=2$}
                \nextgroupplot[width=9cm,height=7cm,domain=2:5,xmode=log,ymode=log, xlabel={$h$}, ylabel={}, title={$k^2=144$}, 
                legend pos=south east, 
                xtick={0.316,0.1,0.0316},
                    ]
                \addplot+[orange,line width=1.5pt,mark=None] table [x=h, y=error, col sep=comma] {error_omega12.0_p1_square.csv}; 
                \addplot+[line width=1.5pt,mark=None, violet] table [x=h, y=error, col sep=comma] {error_omega12.0_p2_square.csv}; 
    
                \draw[orange, dashed,thick] (0.1024,1e-3) -- (0.1024,1e2);
                \draw[violet, dashed,thick] (0.256,1e-3) -- (0.256,1e2);
    
                \addplot[gray,dashed,domain=0.025:0.12] {770.88*x^2};
                \node [draw=none] at (axis description cs:0.25,0.7) {\color{gray}\footnotesize $\!\!\mathcal{O}(h^{2})$};
                
                \addplot[gray,dashed,domain=0.025:0.27] {1500*x^3};
                \node [draw=none] at (axis description cs:0.15,0.3) {\color{gray}\footnotesize $\!\!\mathcal{O}(h^{3})$};
                \nextgroupplot[width=9cm,height=7cm,domain=2:5,xmode=log,ymode=log, xlabel={$h$}, ylabel={$\norm{u-u_h}_{L^2(\Omega)}$}, title={$k^2=225$}, 
                legend pos=south east, 
                xtick={0.316,0.1,0.0316},
                    ]
                \addplot+[orange,line width=1.5pt,mark=None] table [x=h, y=error, col sep=comma] {error_omega15.0_p1_square.csv};
                \addplot+[line width=1.5pt,mark=None, violet] table [x=h, y=error, col sep=comma] {error_omega15.0_p2_square.csv};
                    
                \draw[orange,dashed,thick] (0.0819,1e-3) -- (0.0819,1e3);
                \draw[violet, dashed,thick] (0.2048,1e-3) -- (0.2048,1e3);
    
                \addplot[gray,dashed,domain=0.025:0.1] {1200*x^2};
                \node [draw=none] at (axis description cs:0.25,0.7) {\color{gray}\footnotesize $\!\!\mathcal{O}(h^{2})$};
                
                \addplot[gray,dashed,domain=0.025:0.27] {3000*x^3};
                \node [draw=none] at (axis description cs:0.15,0.35) {\color{gray}\footnotesize $\!\!\mathcal{O}(h^{3})$};
                
                \nextgroupplot[width=9cm,height=7cm,domain=2:5,xmode=log,ymode=log, xlabel={$h$}, ylabel={}, title={$k^2=400$}, 
                legend pos=south east,
                    ]
                \addplot+[orange,line width=1.5pt,mark=None] table [x=h, y=error, col sep=comma] {error_omega20.0_p1_square.csv};
                \addplot+[line width=1.5pt,mark=None, violet] table [x=h, y=error, col sep=comma] {error_omega20.0_p2_square.csv};
    
    
                \draw[orange, dashed,thick] (0.02147,1e-4) -- (0.02147,1e4);
                \draw[violet, dashed,thick] (0.10485,1e-4) -- (0.10485,1e4);
                    
                \addplot[gray,dashed,domain=0.005:0.1] {1250*x^2};
                \node [draw=none] at (axis description cs:0.25,0.55) {\color{gray}\footnotesize $\!\!\mathcal{O}(h^{2})$};
                
                \addplot[gray,dashed,domain=0.005:0.27] {5500*x^3};
                \node [draw=none] at (axis description cs:0.15,0.25) {\color{gray}\footnotesize $\!\!\mathcal{O}(h^{3})$};
                \end{groupplot}
            \end{tikzpicture}
            \pgfplotslegendfromname{named}
        \end{center}
        \caption{$L^2$-error of the approximation of the Helmholtz problem with homogeneous Dirichlet boundary conditions against a computed reference solution, computed with higher polynomial degree, for $k^2 \in \{100,144,225,400\}$. The vertical lines indicate when $\lambda_h^{(i_\ast)} < k^2$ after which we expect quasi-optimality.}
        \label{fig:unitsquareDirichlet}
    \end{figure}
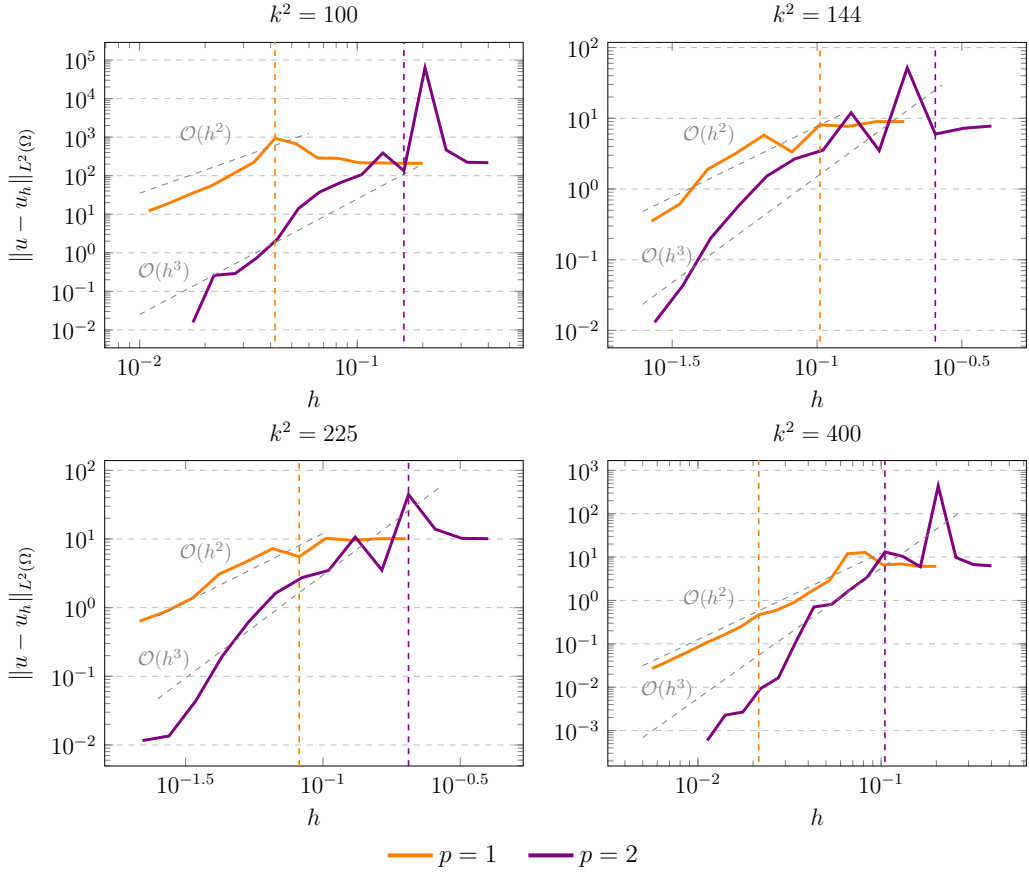
\end{ex}

\subsection{Crouzeix-Raviart finite element discretisation}
We here consider the non-conforming discretisation of \eqref{eq:weakForm} with Crouzeix-Raviart finite elements \cite{CR73}. We denote by $\mathcal{E}$ the set of edges of the triangulation $\mathcal{T}_h$ and define the spaces
\begin{equation*}
    \begin{aligned}
        V_h &:= \{ v_h \in \mathcal{P}^1(\mathcal{T}_h) : v_h \text{ is continuous on } \operatorname{mid}(\mathcal{E}), v = 0 \text{ on } \operatorname{mid}(\mathcal{E} \cap \Gamma_D)\}, \\
        V_{h,0} &:= \{ v_h \in \mathcal{P}^1(\mathcal{T}_h) : v_h \text{ is continuous on } \operatorname{mid}(\mathcal{E}), v = 0 \text{ on } \operatorname{mid}(\mathcal{E} \cap \partial \Omega)\},
    \end{aligned} 
\end{equation*}
where $\mathcal{P}^1(\mathcal{T}_h)$ denotes the space of piecewise linear polynomials on $\mathcal{T}_h$ and $\operatorname{mid}(\mathcal{E}\cap B)$ is the set of midpoints of the edges $\mathcal{E}\cap B \subseteq \mathcal{E}$.
The spaces $V_h$ and $V_{h,0}$ are not $H^1$-conforming, since their continuity is only imposed on the Gauss points on the edges.
We note that a discrete Poincar\'e inequality holds on $V_{h,0}$ and therefore the bilinear form $a_h(\cdot,\cdot)$ is uniformly coercive on $V_{h,0}$ \cite[Lem.~36.6]{EG_FE2}.
In particular, the assumptions on the discrete eigenvalue problem \eqref{eq:discreteEVP} are satisfied \cite{DGP99,Acta10}. Thus, we can apply Thm.~\ref{thm:mainresult} to conclude the well-posedness of the discrete problem \eqref{eq:discreteProblem}.

\begin{ex}[CR vs. $\mathcal{P}^1$]
    In this example we want to compare the approximation with non-conforming Crouzeix-Raviart elements and the approximation with the $H^1$-conforming $\mathcal{P}^1$ element.
    As in the previous examples, we consider the unit square, where the exact eigenvalues of the Laplacian are given by \eqref{eq:unitSquare:EV}.
    In Fig.~\ref{fig:CRvsP1}, we compare the $L^2$-error and the onset of quasi-optimality for both discretisations and increasing wave number.
    We observe that the criterion $\lambda_h^{(i_{\ast})} < k^2 < \lambda_h^{(i_{\ast}+1)}$ indeed guarantees the quasi-optimality of the approximation with Crouzeix-Raviart elements, similarly to what we observed for the $\mathcal{P}^1$-discretisation.
    Furthermore, for all considered wave numbers, the error of the discretisation with Crouzeix-Raviart elements is smaller than the error of the $\mathcal{P}^1$-discretisation once quasi-optimality is achieved.
    
    Depending on the distance between the wave number $k^2$ to the reference eigenvalue $\lambda_h^{(i_{\ast})}$, either the Crouzeix-Raviart discretisation or the $\mathcal{P}^1$ discretisation achieves quasi-optimality first.
    This is due to the fact that the discrete eigenvalues are approximated from above with $\mathcal{P}^1$-elements while in the considered cases Crouzeix-Raviart elements approximate the eigenvalues from below.
    Therefore, if the gap between the wave number $k^2$ and the reference eigenvalue $\lambda^{(i_{\ast})}$ is small, the $\mathcal{P}^1$-approximation requires a small mesh size to ensure that $\lambda_h^{(i_{\ast})} <k^2$ while the Crouzeix-Raviart approximation achieves $\lambda_h^{(i_{\ast}+1)} > k^2$ much sooner.
    In contrast, if $k^2$ is close to $\lambda^{(i_{\ast}+1)}$, the $\mathcal{P}^1$-approximation achieves quasi-optimality sooner than the Crouzeix-Raviart discretisation. Figure \ref{fig:CRvsP1EVs} illustrates the different approximations of the eigenvalues and the required mesh size to ensure that $\lambda_h^{(i_{\ast})} < k^2 < \lambda_h^{(i_{\ast}+1)}$.
    In particular, we observe that the mesh size requirements on the Crouzeix-Raviart discretisation seem to be milder than the mesh size requirements on the $\mathcal{P}^1$-discretisation.
    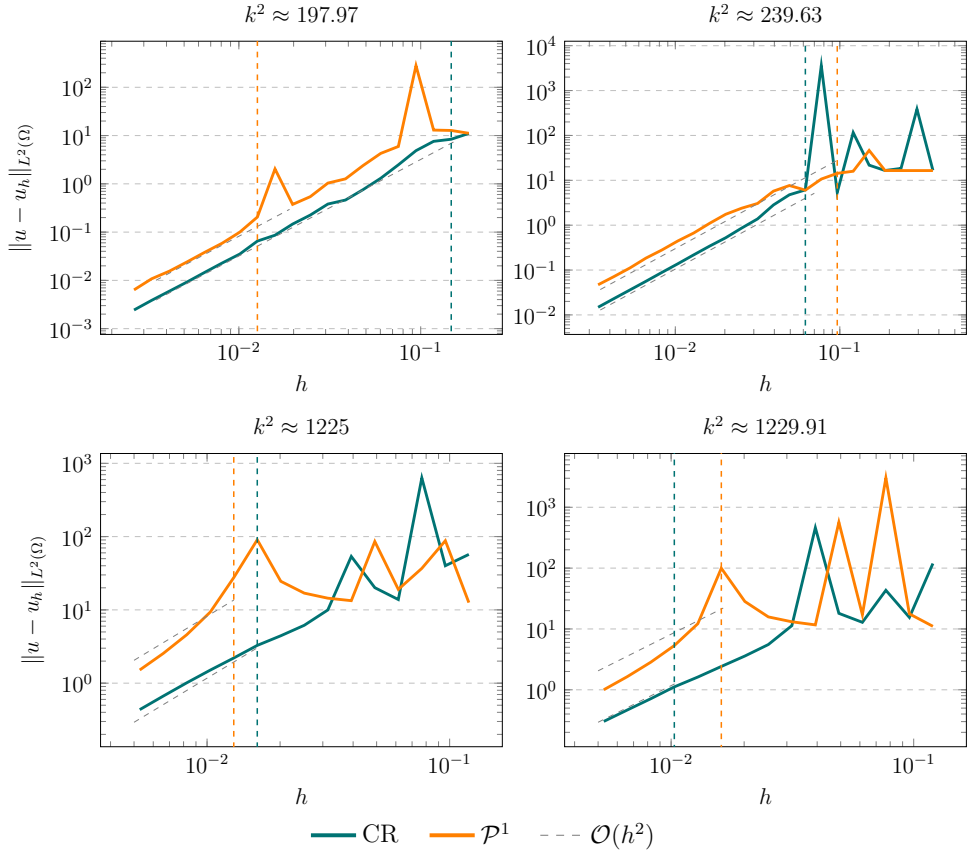
\begin{figure}[!htbp]
        \begin{center}
            \begin{tikzpicture}[scale=0.72]
                \begin{groupplot}[%
                    group style={%
                    group size=2 by 2,
                    horizontal sep=1.15cm,
                    vertical sep=2.2cm,
                    },
                ymajorgrids=true,
                grid style=dashed,
                ]       
                \nextgroupplot[width=9cm,height=7cm,domain=2:5,xmode=log,ymode=log, xlabel={$h$}, ylabel={$\Vert u - u_h \Vert_{L^2(\Omega)}$}, title={$k^2 \approx 197.97$}, 
                legend style={legend columns=3, draw=none,nodes={scale=.8}}, 
                legend to name=named
                    ]

                \addplot+[teal!90!black,line width=1.5pt,mark=None] table [x=h, y=error, col sep=comma] {cr_error_omega14.07_p1.csv};
                \addplot+[orange,line width=1.5pt,mark=None] table [x=h, y=error, col sep=comma] {error_omega14.07_p1.csv};
                    
                \addplot[gray,dashed,domain=0.0035:0.16] {320*x^2};
                \addplot[gray,dashed,domain=0.0035:0.019] {820*x^2};
                \draw[orange, dashed,thick] (0.01264438371942401,1e-4) -- (0.01264438371942401,1e3);
                \draw[teal!90!black, dashed,thick] (0.14720000000000003,1e-4) -- (0.14720000000000003,1e3);

                    
                \legend{CR$\quad$,$\mathcal{P}^1\quad$, $\mathcal{O}(h^2)$}
                \nextgroupplot[width=9cm,height=7cm,domain=2:5,xmode=log,ymode=log, xlabel={$h$}, ylabel={}, title={$k^2 \approx 239.63$}, 
                    ]

                \addplot+[teal!90!black,line width=1.5pt,mark=None] table [x=h, y=error, col sep=comma] {cr_error_omega15.48_p1_temp.csv};
                \addplot+[orange,line width=1.5pt,mark=None] table [x=h, y=error, col sep=comma] {error_omega15.48_p1_temp.csv};
                \addplot[gray,dashed,domain=0.0035:0.07] {1050*x^2};
                \addplot[gray,dashed,domain=0.0035:0.1] {3000*x^2};

                \draw[orange, dashed,thick] (0.09646899200000005,1e-3) -- (0.09646899200000005,1e4);
                \draw[teal!90!black, dashed,thick] (0.06174015488000004,1e-3) -- (0.06174015488000004,1e4);

                \nextgroupplot[width=9cm,height=7cm,domain=2:5,xmode=log,ymode=log, xlabel={$h$}, ylabel={$\Vert u - u_h \Vert_{L^2(\Omega)}$}, title={$k^2 \approx 1225$}, 
                    ]

                \addplot+[teal!90!black,line width=1.5pt,mark=None] table [x=h, y=error, col sep=comma] {cr_error_omega35.0_p1.csv};
                \addplot+[orange,line width=1.5pt,mark=None] table [x=h, y=error, col sep=comma] {error_omega35.0_p1.csv};
                
                \addplot[gray,dashed,domain=0.005:0.0162] {11700*x^2};
                \addplot[gray,dashed,domain=0.005:0.0129] {82000*x^2};

                \draw[orange, dashed,thick] (0.01288490188800001,1e-4) -- (0.01288490188800001,1e4);
                \draw[teal!90!black, dashed,thick] (0.01610612736000001,1e-4) -- (0.01610612736000001,1e4);

                \nextgroupplot[width=9cm,height=7cm,domain=2:5,xmode=log,ymode=log, xlabel={$h$}, ylabel={}, title={$k^2 \approx 1229.91$}, 
                    ]

                \addplot+[teal!90!black,line width=1.5pt,mark=None] table [x=h, y=error, col sep=comma] {cr_error_omega35.07_p1.csv};
                \addplot+[orange,line width=1.5pt,mark=None] table [x=h, y=error, col sep=comma] {error_omega35.07_p1.csv};

                \addplot[gray,dashed,domain=0.005:0.0104] {11700*x^2};
                \addplot[gray,dashed,domain=0.005:0.0164] {82000*x^2};

                \draw[orange, dashed,thick] (0.01610612736000001,1e-4) -- (0.01610612736000001,1e4);
                \draw[teal!90!black, dashed,thick] (0.010307921510400008,1e-4) -- (0.010307921510400008,1e4);

                \end{groupplot}
            \end{tikzpicture}
            \pgfplotslegendfromname{named}
        \end{center}\vspace*{-0.1cm}
        \caption{We compare the approximation error using Crouzeix-Raviart elements and conforming $\mathcal{P}^1$-elements.
        For increasing wave number $k^2$, we compute the $L^2$-error against a reference solution computed with a higher order $H^1$-conforming method.
        The dashed lines indicate the mesh size at which the criterion $\lambda_h^{(i_{\ast})} < k^2$ (for the $\mathcal{P}^1$-case) or $\lambda_h^{(i_{\ast}+1)} > k^2$ (for the Crouzeix-Raviart case).
        For the first two wave-numbers (top row) the relevant eigenvalues are $\lambda^{(i_{\ast})} = 197.39$ and $\lambda^{(i_{\ast}+1)} = 246.74$, whereas for the other wave-numbers (bottom row) the relevant eigenvalues are $\lambda^{(i_{\ast})} = 1224.09$ and $\lambda^{(i_{\ast}+1)} = 1233.70$.}
        \label{fig:CRvsP1}
    \end{figure}
 
    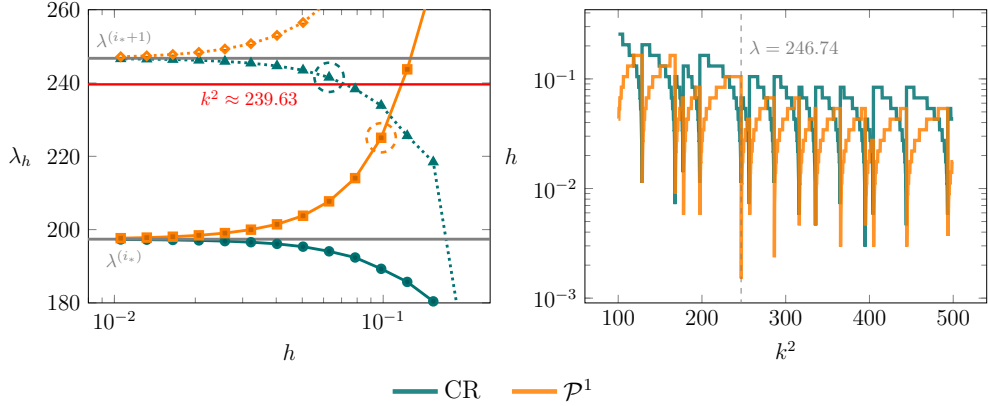
\begin{figure}[!htbp]
        \begin{center}
            \begin{tikzpicture}[scale=0.72]
                \begin{groupplot}[%
                    group style={%
                    group size=2 by 1,
                    horizontal sep=1.75cm,
                    vertical sep=2cm,
                    }, 
                ]
                \nextgroupplot[width=9cm,height=7cm,domain=2:5,xmode=log,ymode=linear, xlabel={$h$}, ylabel={\rotatebox{270}{$\lambda_h$}}, title={}, 
                ymin = 180, ymax = 260, xmin=0.8e-2, xmax=2.5e-1,
                    ]
        
                \addplot+[teal!90!black,line width=1.5pt] table [x=h, y=EV_i, col sep=comma] {cr_error_omega15.48_p1.csv};
                \addplot+[orange,line width=1.5pt] table [x=h, y=EV_i, col sep=comma] {error_omega15.48_p1.csv};

                \addplot+[teal!90!black,line width=1.5pt,dotted] table [x=h, y=EV_ipo, col sep=comma] {cr_error_omega15.48_p1.csv};
                \addplot+[orange,line width=1.5pt,dotted] table [x=h, y=EV_ipo, col sep=comma] {error_omega15.48_p1.csv};

                \addplot+[gray,line width=1.5pt,mark=None,domain=1e-5:0.25] {197.39};
                \addplot+[gray,line width=1.5pt,mark=None,domain=1e-5:0.25] {246.74};
                \node at (axis description cs:0.09,0.15) {\color{gray}\footnotesize $\lambda^{(i_{\ast})}$};
                \node at (axis description cs:0.09,0.9) {\color{gray}\footnotesize $\lambda^{(i_{\ast}+1)}$};

                \draw[draw=orange,fill=none,line width=1.35pt,dashed] (0.09830400000000002,225.05637163188018) circle[radius=0.27cm];
                \draw[draw=teal!90!black,fill=none,line width=1.35pt,dashed] (0.06291456000000002,241.54226257199525) circle[radius=0.27cm];
                \addplot+[red,line width=1.25pt,mark=None,domain=1e-5:0.25] {15.48*15.48};
                \node at (axis description cs: 0.4,0.7) {\color{red} \footnotesize $k^2 \approx 239.63$};



                \nextgroupplot[width=9cm,height=7cm,domain=2:5,xmode=linear,ymode=log, xlabel={$k^2$}, ylabel={\rotatebox{270}{$h$}}, title={}, 
                legend style={legend columns=2, draw=none,nodes={scale=.8}}, 
                legend to name=named
                    ]

                \addplot+[teal!90!black,line width=1.75pt,mark=None,opacity=0.85] table [x=wavenumbers, y=hCR, col sep=comma] {CRvsP1h_100.0_to_500.0.csv};
                \addplot+[orange,line width=1.75pt,mark=None,opacity=0.85] table [x=wavenumbers, y=hP1, col sep=comma] {CRvsP1h_100.0_to_500.0.csv};

                \draw[gray,dashed] (246.74,1e-4) -- (246.74,1e0);
                \node [draw=none] at (axis description cs:0.52,0.87) {\color{gray}\footnotesize $\lambda = 246.74$};

                \legend{CR$\quad$,$\mathcal{P}^1$}

                \end{groupplot}
            \end{tikzpicture} \\
            \pgfplotslegendfromname{named}
        \end{center}\vspace*{-0.1cm}
        \caption{On the left, we compare the approximation of the eigenvalues $\lambda^{(i)} = 197.39$ (solid) and $\lambda^{(i+1)} = 246.74$ (dashed) with Crouzeix-Raviart and $\mathcal{P}^1$-elements. For $k^2 \approx 239.63$, cf.~Figure \ref{fig:CRvsP1}, the $\mathcal{P}^1$-discretisation reaches quasi-optimality (dashed circle) before the Crouzeix-Raviart discretisation.
        On the right,  we compare mesh size required for the criterion $\lambda_h^{(i_{\ast})} < k^2 < \lambda_h^{(i_{\ast}+1)}$ to be met with the respective approximations for wave numbers $k^2\in [100,500]$. When the reference eigenvalue $\lambda^{(i_{\ast})}$ changes, the requirements on the mesh size become suddenly stricter for the $\mathcal{P}^1$-discretisation and more relaxed for the Crouzeix-Raviart discretisation, consider for example the moment the wave number becomes larger than the eigenvalue $\lambda = 246.74$.}
        \label{fig:CRvsP1EVs}
    \end{figure}
\end{ex}

While the approximation of the eigenvalues with a conforming $\mathcal{P}^1$-element is monotone and converges from above, the Crouzeix--Raviart element are much more interesting in this regard.
In fact, Crouzeix--Raviart discretisations of the Lapalcian can approximate the eigenvalues of the operator from both above and below. Yet some simple post-processing techniques can be used to obtain guaranteed lower bounds for the eigenvalues.
\begin{proposition}[Guaranteed lower bounds \cite{CG14}]
    Let us consider the discrete eigenvalue problem \eqref{eq:discreteEVP} with Crouzeix--Raviart elements.
    We will here assume the numerical linear algebraic problem of computing the eigenvalues is solved exactly and that for a fixed $j\in \mathbb{N}$
    \begin{equation}
        \label{eq:separation}
        h\leq \frac{\sqrt{1+\frac{1}{j}}-1}{\kappa \sqrt{\lambda^{(j)}}},
    \end{equation}
    where $h$ is the mesh size of triangulation $\mathcal{T}_h$ and $\kappa$ is a constant related to the Poincar\'e inequality on a single element, here taken to be $\kappa \leq 0.1932$.
    Then the following bound must also hold
    \begin{equation}
        \underline{\lambda}_h^{(j)}\coloneqq \frac{\lambda_h^{(j)}}{1+\kappa^2\lambda_h^{(j)}h^2} \leq \lambda^{(j)},
    \end{equation}
\end{proposition}
Following the same idea presented in \cite{CG14}, we can compute an upper bound for the eigenvalues by interpolating the eigenfunctions $e_h^{(j)}$ on the $H^1$-conformingm $\mathcal{P}^1(\mathcal{T}_h)$ space and computing the Rayleigh quotient of the interpolated function.
We will here denote such upper bound as $\overline{\lambda}_h^{(j)}$.
While at first the separation condition \eqref{eq:separation} might seem stringent, we can observe that we are only interested in the eigenvalues $\lambda^{(j)}$ for which $\lambda^{(j)}<k^2$, hence the separation condition is verified if 
\begin{equation}
    h \leq \frac{\sqrt{1+\frac{1}{j}}-1}{\kappa}k^{-1}\sim\frac{\sqrt{1+(dk)^{-1}}-1}{\kappa}k^{-1}\sim k^{-\frac{3}{2}},
\end{equation}
where we used Weyl's inequality to obtain the last identity.
Therefore \eqref{eq:separation} is equivalent to ask that the mesh size decreases as $k^{-\frac{3}{2}}$, which is a reasonable assumption for the mesh size in the context of the Helmholtz equation, \cite{GS24}.
\section{Guaranteeing quasi-optimality via mesh refinement}
\label{sec:gmr}
\noindent
In this section, we present a simple algorithm that, for a given coarse mesh $\mathcal{T}_h^{(0)}$ and a wave number $k^2$, constructs a mesh $\mathcal{T}_h$ such that the finite element approximation of the Helmholtz problem is quasi-optimal.
\begin{algorithm}[h]
\caption{Guaranteed mesh refinement}\label{alg:gmr}
\begin{algorithmic}
\Require A triangulation $\mathcal{T}_h^{(0)}$ of $\Omega$, a FEM degree $p$, a wave number $k$
\Ensure $k^2 \leq \lambda_h^{(i_{\ast})}$ and $k^2 \geq \lambda_h^{(i_{\ast}+1)}$
\State $\mathcal{T}_h \gets \text{REFINE}(\mathcal{T}_h^{(0)})$
\State $\lambda_h^{(i_\ast)} \gets \text{ESTIMATE}(\mathcal{T}_h,p)$
\While{$k^2 \leq \lambda_h^{(i_\ast)}$ and $k^2 \geq \lambda_h^{(i_\ast+1)}$}
\State $\mathcal{T}_h \gets \text{REFINE}(\mathcal{T}_h)$
\State $\lambda_h^{(i_\ast)} \gets \text{ESTIMATE}(\mathcal{T}_h,p)$
\EndWhile
\end{algorithmic}
\end{algorithm}
The algorithm is based on the condition $\lambda_h^{(i_{\ast})} < k^2 < \lambda_h^{(i_\ast+1)}$ derived in Thm. \ref{thm:mainresult}.
It is clear from Algorithm \ref{alg:gmr} that the key ingredients of the scheme here proposed are the REFINE and ESTIMATE steps, here discussed in greater details.\\[0.25cm]
\noindent
\emph{REFINE}, The refinement strategy is a key ingredient in the proposed mesh refinement strategy. The goal is to refine the mesh until the condition $\lambda_h^{(i_{\ast})} < k^2 < \lambda_h^{(i_\ast+1)}$ is satisfied.
The simplest refinement strategy is the uniform refinement, where each element is divided into four equal sub-elements. This strategy is simple to implement and guarantees that the condition $\lambda_h^{(i_{\ast})} < k^2 < \lambda_h^{(i_\ast+1)}$ will be satisfied for sufficiently fine meshes.
Uniform mesh refinement is explored in Examples \ref{ex:tuningfork} and \ref{ex:scatterer}.
However, uniform refinement is not optimal and can lead to a large number of elements.
Alternatively, one can use an adaptive mesh refinement strategy. In this paper we investigate a naive adaptive refinement strategy based on the error estimator $\eta$ defined as 
\begin{equation}
    \label{eq:eta}
    \eta = i_{\ast}^{-1}\sum_{i=1}^{i_\ast+\ell} \! \sum_{K \in \mathcal{T}_h} \! \! \left(h^2_K \norm{\Delta e_h^{(i)}+\lambda_h^{(i)} e_h^{(i)}}_{L^2(K)}^2 + \frac{h_K}{2}\norm{\nabla e_h^{(i)}\cdot \bm{n}}_{L^2(\partial K\backslash \partial\Omega)}^2\right),
\end{equation}
where $\ell$ is a small integer, $e_h^{(i)}$ and $\lambda_h^{(i)}$ are the $i$-th eigenfunction and eigenvalue, $h_K$ is the diameter of the element $K$, and $\bm{n}$ is the outward unit normal to $\partial K$.
This error estimator is essentially \textit{Babu\v{s}ka–Rheinboldt estimator}, averaged over the first $i_{\ast}+\ell$ eigenfunctions \cite{BR78}.
More sophisticated adaptive mesh refinement strategies can be used, we refer the reader interested in this topic to \cite{N09}.
We explore the naive adaptive refinement strategy based on the error estimator introduced in \eqref{eq:eta} in Example \ref{ex:plane_babuska}.
\\[0.25cm]
\noindent
\emph{ESTIMATE},  The estimate of the eigenvalue $\lambda_h^{(i_{\ast})}$ is critical, in fact while checking the condition $\lambda_h^{(i_{\ast})} < k^2 < \lambda_h^{(i_{\ast}+1)}$ is simple, at the same time we need to compute the correct index $i_{\ast}$.
The index $i_{\ast}$ might be already known since a modal analysis has been conducted on the domain to ensure that the wave-number $k$ is not a resonance frequency.

If we are using a Crouzeix-Raviart discretisation another option is to estimate the index $i_{\ast}$ using the guaranteed lower bounds for the eigenvalues $\underline{\lambda}_h^{(j)}$.
For a fixed a maximum index $j_{\max}$, on the coarsest mesh $\mathcal{T}_h^{(0)}$ we compute the first $j_{\max}$ eigenvalues.
We then compute the guaranteed lower bounds $\underline{\lambda}_h^{(j)}$ for $j=1,\ldots,j_{\max}$ and fix as guess for the index $i_{\ast}$ the largest $j$ such that $\underline{\lambda}_h^{(j+1)}\geq k^2$. We denote such index as $j_{\ast}$.
We can then proceed to adaptively refine the mesh using as indicator the error estimator $\eta$ defined similarly to \eqref{eq:eta} but with $i_{\ast}$ replaced by $j_{\ast}$.
Each time we refine the mesh we can update the index $j_{\ast}$ by computing the guaranteed lower bounds for the eigenvalues on the new mesh and fixing as new index $j_{\ast}$ the largest $j$ such that $\underline{\lambda}_h^{(j+1)}\geq k^2$.
Using the guaranteed lower bounds for the eigenvalues we can ensure that $i_{\ast}$ is correctly estimated. In fact if $k^2-\lambda_h^{(j_{\ast})}$ is less than $\overline{\lambda}_h^{(j_{\ast})}-\underline{\lambda}_h^{(j_{\ast})}$ we know the index $j_{\ast}$ can not change anymore, hence we have correctly estimated $i_{\ast}$.
\\[0.4cm]

\begin{remark}
    While we here focused on the use of the guaranteed lower bounds only for Crouzeix-Raviart discretisations, the same idea can be applied to $H^1$-conforming discretisations using the guaranteed lower bounds presented in \cite{CDM17}.
\end{remark}

\begin{ex}[The tuning fork]
\label{ex:tuningfork}
We here consider a more challenging geometry: the tuning fork domain.
In Figure \ref{fig:tuningforkDirichlet} we show the results of the numerical experiments for the Helmholtz problem with $k^2=100$ and  Dirichlet boundary conditions.
In particular, we consider as data the Gaussian bump 
\begin{equation}
    \label{eq:GB}
    f(x,y) = 5 \cdot 10^4 \exp \left[{-(40^2)\cdot ((x-\mathcal{O}_x)^2+(y-\mathcal{O}_y)^2)}\right], \quad \mathcal{O}=(10,7/4).
\end{equation}
We solved the Helmholtz problem, using conforming a $\mathcal{P}^1$ conforming finite element method on a sequence of uniformly refined meshes until $k^2-\lambda^{(i_\ast)}_h$ was positive.
We then solved the problem on one additional uniform refinement for comparison.
The value of $k^2-\lambda^{(i_\ast)}_h$ for the different meshes is shown in Table \ref{tab:tuningforkDirichlet}.
Before the condition $k^2-\lambda^{(i_\ast)}_h>0$ was satisfied we observe in Figure \ref{fig:tuningforkDirichlet} the presence of a spurious symmetric eigen mode in poorly approximated Helmholtz solutions.
This results in the solution of the Helmholtz problem being symmetric with respect to the $x$-axis, for poorly approximated solutions.
\begin{figure}
\centering 
    \begin{subfigure}{0.45\linewidth}
        \includegraphics[width=\linewidth]{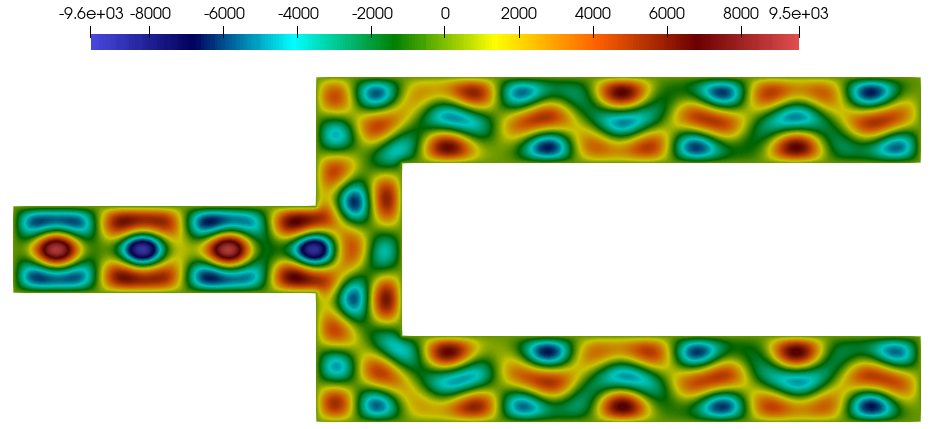} 
        \subcaption{$Number of DoFs:  21521$.}
    \end{subfigure}\hfill
    \begin{subfigure}{0.45\linewidth}
        \includegraphics[width=\linewidth]{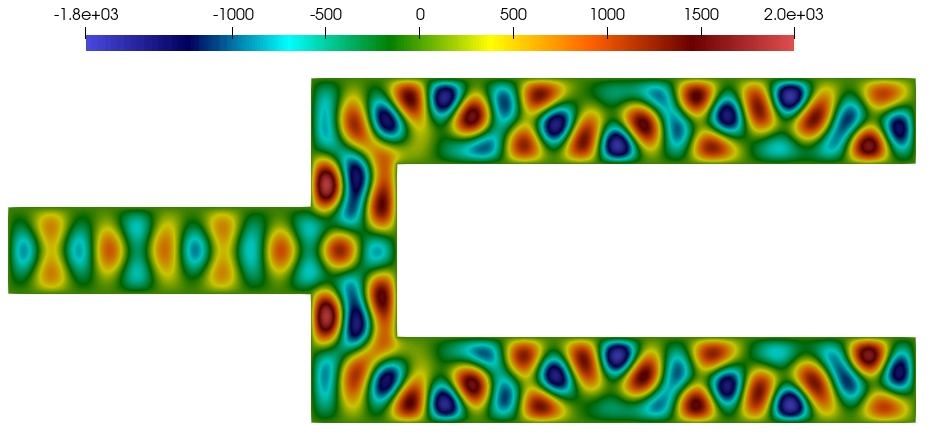} 
        \subcaption{$Number of DoFs:  84769$.}
    \end{subfigure}
    \begin{subfigure}{0.45\linewidth}
        \includegraphics[width=\linewidth]{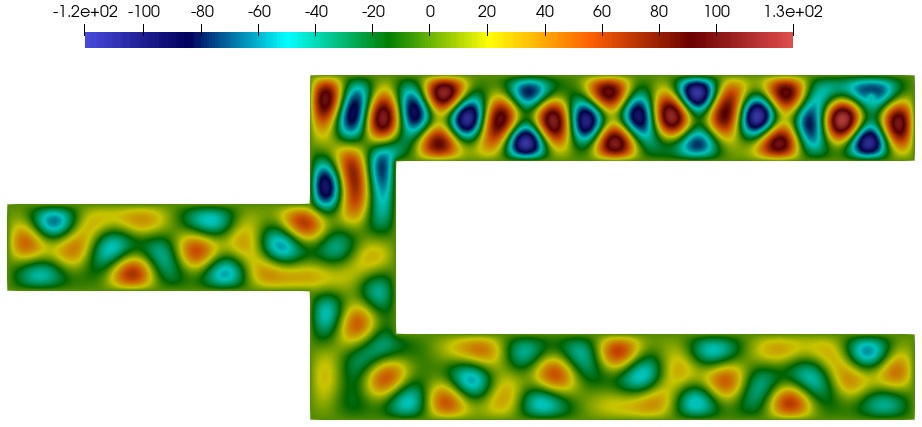} 
        \subcaption{$Number of DoFs:  336449$.}
    \end{subfigure}\hfill
    \begin{subfigure}{0.45\linewidth}
        \includegraphics[width=\linewidth]{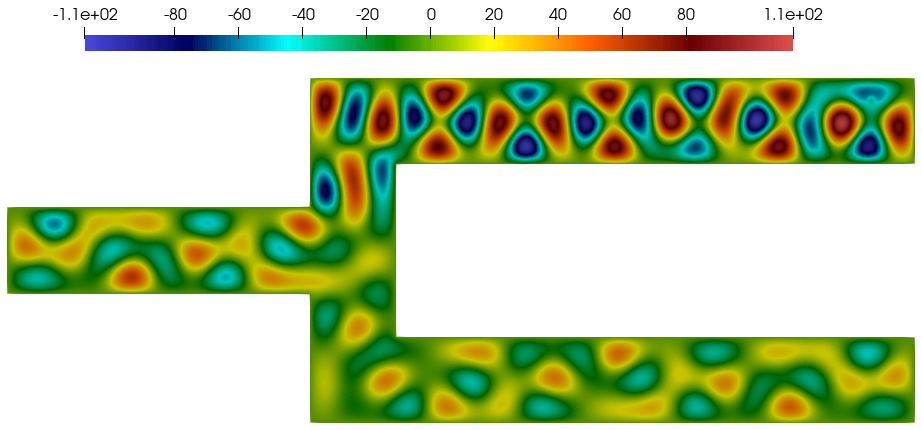} 
        \subcaption{$Number of DoFs:  1340545$.}
    \end{subfigure}
\caption{We solve the Helmholtz problem on a tuning fork domain using a sequence of uniformly refined meshes, with $k^2$ equal to $100$ and data a Gaussian bump $f(x,y) = 5\times 10^4 \exp \left[{-(40^2)\cdot ((x-10)^2+(y-\frac{7}{4})^2)}\right]$.
        The criterion here proposed suggest that mesh $(c)$ is the first one to guarantee quasi-optimality.}
\label{fig:tuningforkDirichlet}
\end{figure}
\begin{table}[!htbp]
    \caption{The table shows $k^2-\lambda^{(i_\ast)}_h$ for different degrees of freedom corresponding to meshes obtained by uniform refinement, in the setting of Example \ref{ex:tuningfork}.}
    \label{tab:tuningforkDirichlet}
    \pgfplotstableread[
    col sep=comma,
    ]{results_ding.csv}\normal

    \pgfplotstabletranspose[
        colnames from=ndof,
        columns={ndof, condition},
    ]\transpose\normal

    \begin{center}
        \pgfplotstabletypeset[
            string type,
            every head row/.style={before row=\toprule, after row=\midrule},
            every last row/.style={after row=\bottomrule},
            every col no 0/.style={
                column type={l},
                column name={N. DoFs},
                postprddoc cell content/.append style={/pgfplots/table/@cell content/.add={}{}},
            },
            every row 0 column 0/.style={postproc cell content/.style={@cell
                content=$k^2 - \lambda^{(i_{\ast})}$}},
        ]\transpose
    \end{center}
\end{table}
\end{ex}
\begin{ex}[Scatterer]
    \label{ex:scatterer}
    We consider the Helmholtz problem on a domain with non-trivial topology, represented by a scatterer.
    We solve the Helmholtz problem with $k^2=1600$ and homogeneous Neumann boundary conditions, on the boundary of the outer circular domain. 
    We used a $\mathcal{P}^1$ conforming finite element method.
    The data is given by the Gaussian bump similar to \eqref{eq:GB} and centered at $\mathcal{O}=(-\frac{1}{2}, -\frac{1}{2})$.
    We kept refining the mesh uniformly until $k^2-\lambda^{(i_\ast)}_h$ was positive.
    The value of $k^2-\lambda^{(i_\ast)}_h$ for the different meshes is shown in Table \ref{tab:tuningScatterer}.
    Before the condition $k^2-\lambda^{(i_\ast)}_h>0$ was satisfied we clearly see inspecting Figure \eqref{fig:tuningScatterer} that we have non-symmetric solution, while after the condition is satisfied we have a symmetric solution as expected due to the symmetry of \eqref{eq:GB} in the scatterer domain.
    \begin{figure}
\centering
    \begin{subfigure}{0.45\linewidth}
        \includegraphics[width=\linewidth]{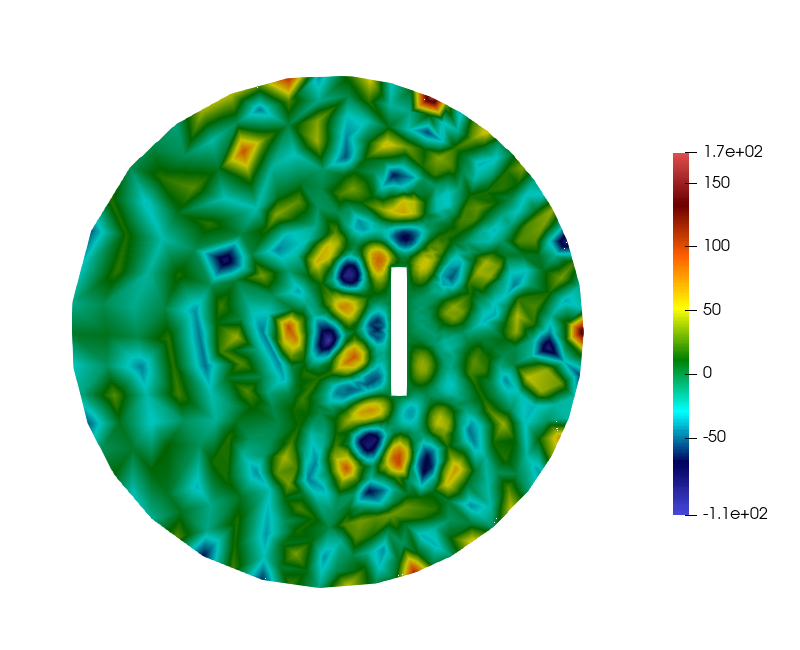} 
        \subcaption{$Number of DoFs:  1160$.}
    \end{subfigure}\hfill
    \begin{subfigure}{0.45\linewidth}
        \includegraphics[width=\linewidth]{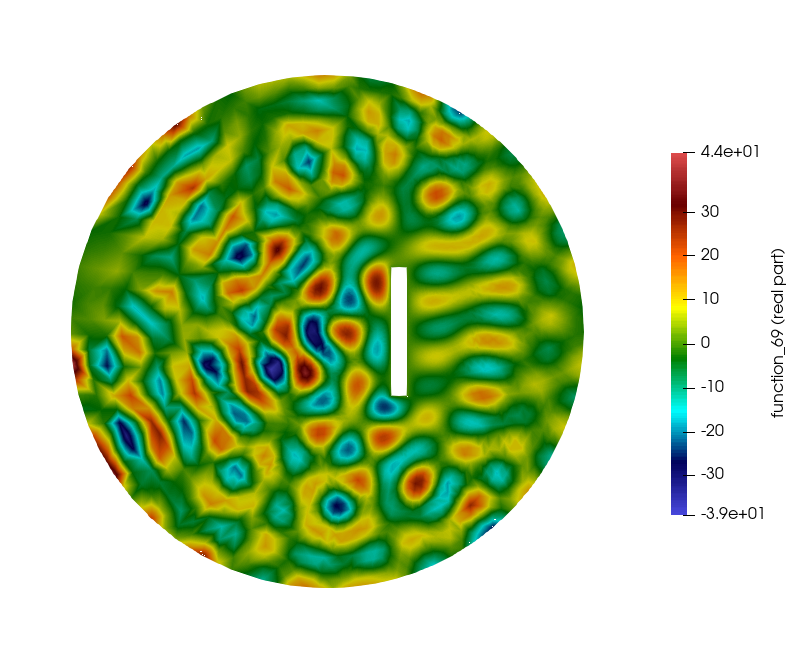} 
        \subcaption{$Number of DoFs:  4528$.}
    \end{subfigure}
    \begin{subfigure}{0.45\linewidth}
        \includegraphics[width=\linewidth]{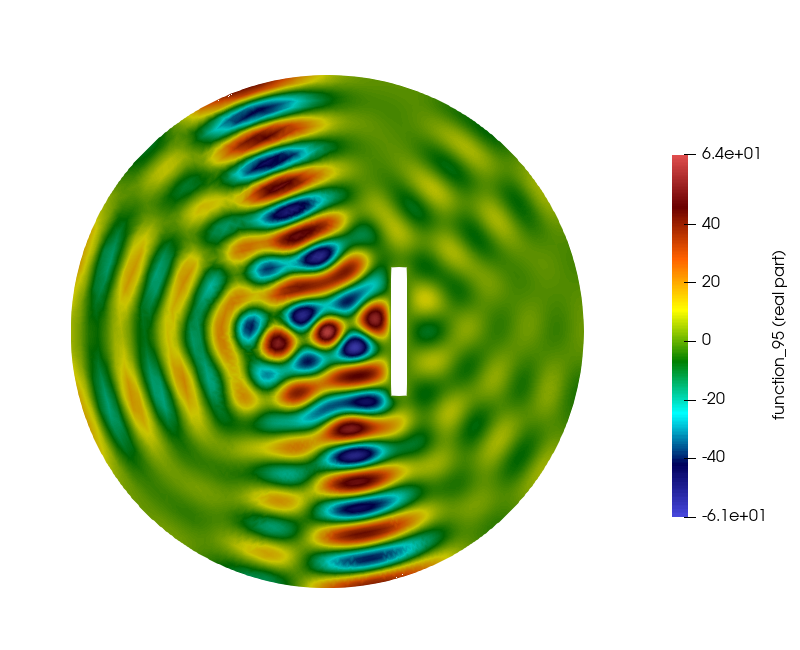} 
        \subcaption{$Number of DoFs: 17888$.}
    \end{subfigure}\hfill
    \begin{subfigure}{0.45\linewidth}
        \includegraphics[width=\linewidth]{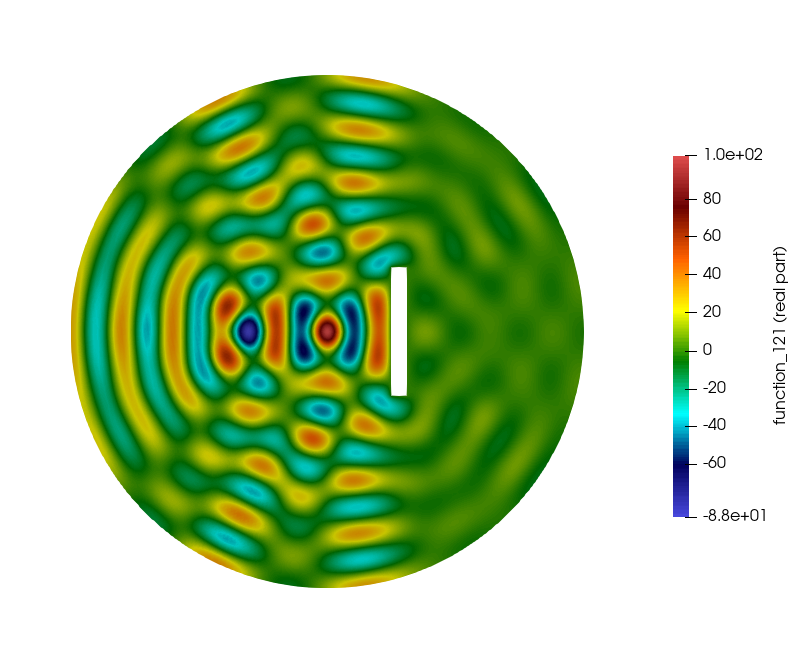} 
        \subcaption{$Number of DoFs:  71104$.}
    \end{subfigure}
\caption{We solve the Helmholtz problem on a tuning fork domain using a sequence of uniformly refined meshes, with $k$ equal to $10$ and data a Gaussian bump $f(x,y) = 5\times 10^4 \exp \left[{-(40^2)\cdot ((x+\frac{1}{2})^2+(y+\frac{1}{2})^2)}\right]$.
        The criterion here proposed suggest that mesh $(c)$ is the first one to guarantee quasi-optimality.}
\label{fig:tuningScatterer}
\end{figure}
\begin{table}[!htbp]
    \caption{The table shows $k^2-\lambda^{(i_\ast)}_h$ for different degrees of freedom corresponding to meshes obtained by uniform refinement, in the setting of Example \ref{ex:scatterer}.}
    \label{tab:tuningScatterer}
    \pgfplotstableread[
    col sep=comma,
    ]{results_scatterer.csv}\normal

    \pgfplotstabletranspose[
        colnames from=ndof,
        columns={ndof, condition},
    ]\transpose\normal

    \begin{center}
        \pgfplotstabletypeset[
            string type,
            every head row/.style={before row=\toprule, after row=\midrule},
            every last row/.style={after row=\bottomrule},
            every col no 0/.style={
                column type={l},
                column name={N. DoFs},
                postprddoc cell content/.append style={/pgfplots/table/@cell content/.add={}{}},
            },
            every row 0 column 0/.style={postproc cell content/.style={@cell
                content=$k^2 - \lambda^{(i_{\ast})}$}},
        ]\transpose
    \end{center}
\end{table}
\end{ex}
\begin{ex}(Adaptivity)
    \label{ex:plane_babuska}
    Once again we consider the Helmholtz problem on a domain with non-trivial topology.
    We solve the Helmholtz problem with $k^2=400$ and homogeneous Dirichlet boundary conditions.
    We begin considering a $\mathcal{P}^1$ conforming finite element method and consider as data the Gaussian bump \eqref{eq:GB} centered at $\mathcal{O}=(\frac{3}{4},-\frac{3}{4})$.
    We kept refining the mesh until $k^2-\lambda^{(i_\ast)}_h$ was positive. To compute the index $i_{\ast}$ we used a reference solution computed with a higher order finite element method on a much finer mesh.
    In particular, we used the error estimator $\eta$ defined in \eqref{eq:eta} to mark for refinement elements whose value $\eta$ was above half the maximum value of $\eta$.
    The value of $k^2-\lambda^{(i_\ast)}_h$ for the different meshes is shown in Table \ref{tab:adptCG}.
    \begin{table}[!htbp]
        \caption{The table shows $k^2-\lambda^{(i_\ast)}_h$ for different degrees of freedom corresponding to meshes obtained by adaptive refinement, based on the indicator defined in \ref{eq:eta}, in the setting of Example \ref{ex:plane_babuska}.
        A conforming $\mathcal{P}^1$-finite element method is used to solve the Helmholtz problem.}
        \label{tab:adptCG}
        \pgfplotstableread[
        col sep=comma,
        ]{results_plane_cg.csv}\normal

        \pgfplotstabletranspose[
            colnames from=ndof,
            columns={ndof, condition},
        ]\transpose\normal

        \begin{center}
            \pgfplotstabletypeset[
                string type,
                every head row/.style={before row=\toprule, after row=\midrule},
                every last row/.style={after row=\bottomrule},
                every col no 0/.style={
                    column type={l},
                    column name={N. DoFs},
                    postprddoc cell content/.append style={/pgfplots/table/@cell content/.add={}{}},
                },
                every row 0 column 0/.style={postproc cell content/.style={@cell
                content=$k^2 - \lambda^{(i_{\ast})}$}},
            ]\transpose
        \end{center}
    \end{table}
    We then solve the same problem using the Crouzeix-Raviart finite element method.
    Once again we used the error estimator $\eta$ defined in \eqref{eq:eta} to mark for refinement elements whose value $\eta$ was above half the maximum value of $\eta$.
    Using the estimate discussed in Section \ref{sec:gmr} we the index $i_{\ast}$ was estimated by $j_{\ast}$.
    We kept refining the mesh until $k^2-\lambda^{(j_\ast)}_h$ was positive and we can certify that the index $j_{\ast}$ is the same as $i_{\ast}$, via the granted lower bound estimates discussed in Section \ref{sec:gmr}.
    \begin{table}[!htbp]
        \caption{The table shows $k^2-\lambda^{(j_\ast)}_h$ and $j_{\ast}$ for different degrees of freedom corresponding to meshes obtained by adaptive refinement, based on the indicator defined in \ref{eq:eta}, in the setting of Example \ref{ex:plane_babuska}.
        A Crouzeix-Raviart finite element method is used to solve the Helmholtz problem.
        We also report the difference between the upper and lower estimates of the eigenvalue $\lambda_h^{(j_{\ast})}$ and the estimated index $j_{\ast}$.
        While the index $j_{\ast}$ is correctly estimated on coarser meshes and the cirterion $k^2-\lambda^{(j_{\ast})}_h>0$ is satisfied on the same mesh, only on a much finer mesh we are able to certify that the index $j_{\ast}$ is correctly estimated, since $\overline{\lambda}_h^{(j_{\ast})}-\underline{\lambda}_h^{(j_{\ast})}$ is smaller than $k^2-\lambda^{(j_{\ast})}_h\approx 3.88$.}
        \label{tab:adptCR}
        \pgfplotstableread[
        col sep=comma,
        ]{results_plane_cr1.csv}\normal

        \pgfplotstabletranspose[
            colnames from=ndof,
            columns={ndof, condition,eta, estIndex},
        ]\transpose\normal

        \begin{center}
            \pgfplotstabletypeset[
                string type,
                every head row/.style={before row=\toprule, after row=\midrule},
                every last row/.style={after row=\bottomrule},
                every col no 0/.style={
                    column type={l},
                    column name={N. DoFs},
                    postprddoc cell content/.append style={/pgfplots/table/@cell content/.add={}{}},
                },
                every row 0 column 0/.style={postproc cell content/.style={@cell
                content=$k^2 - \lambda^{(j_{\ast})}_h$}},
                every row 1 column 0/.style={postproc cell content/.style={@cell
                content=$\overline{\lambda}_h^{(j_\ast)}-\underline{\lambda}_h^{(j_\ast)}$}},
                every row 2 column 0/.style={postproc cell content/.style={@cell
                content=$j_\ast$}},
            ]\transpose
        \end{center}
        \pgfplotstableread[
        col sep=comma,
        ]{results_plane_cr2.csv}\normal

        \pgfplotstabletranspose[
            colnames from=ndof,
            columns={ndof, condition, eta, estIndex},
        ]\transpose\normal

        \begin{center}
            \pgfplotstabletypeset[
                string type,
                every head row/.style={before row=\toprule, after row=\midrule},
                every last row/.style={after row=\bottomrule},
                every col no 0/.style={
                    column type={l},
                    column name={N. DoFs},
                    postprddoc cell content/.append style={/pgfplots/table/@cell content/.add={}{}},
                },
                every row 0 column 0/.style={postproc cell content/.style={@cell
                content=$k^2 - \lambda^{(j_{\ast})}_h$}},
                every row 1 column 0/.style={postproc cell content/.style={@cell
                content=$\overline{\lambda}_h^{(j_\ast)}-\underline{\lambda}_h^{(j_\ast)}$}},
                every row 2 column 0/.style={postproc cell content/.style={@cell
                content=$j_\ast$}},
            ]\transpose
        \end{center}
    \end{table}
    As we can see from Figure \ref{fig:plane_adpt_mesh} the mesh obtained by adaptive refinement for both the $\mathcal{P}^1$-conforming finite element method and the Crouzeix-Raviart finite element method is much more refined near the corner of the inner domain, as expected.
    \begin{figure}[h]
        \centering 
        \includegraphics[scale=0.12]{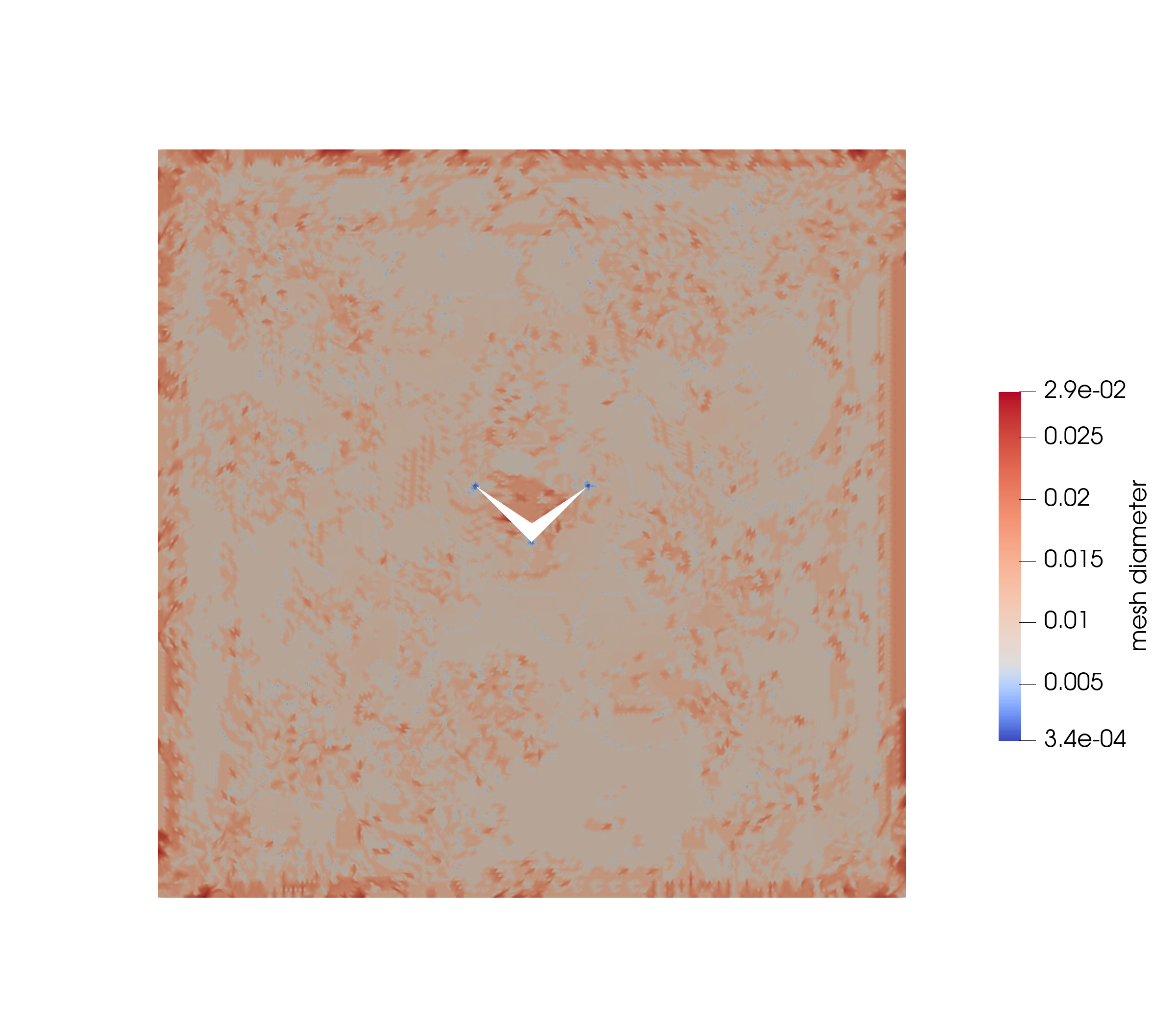}  
        \includegraphics[scale=0.12]{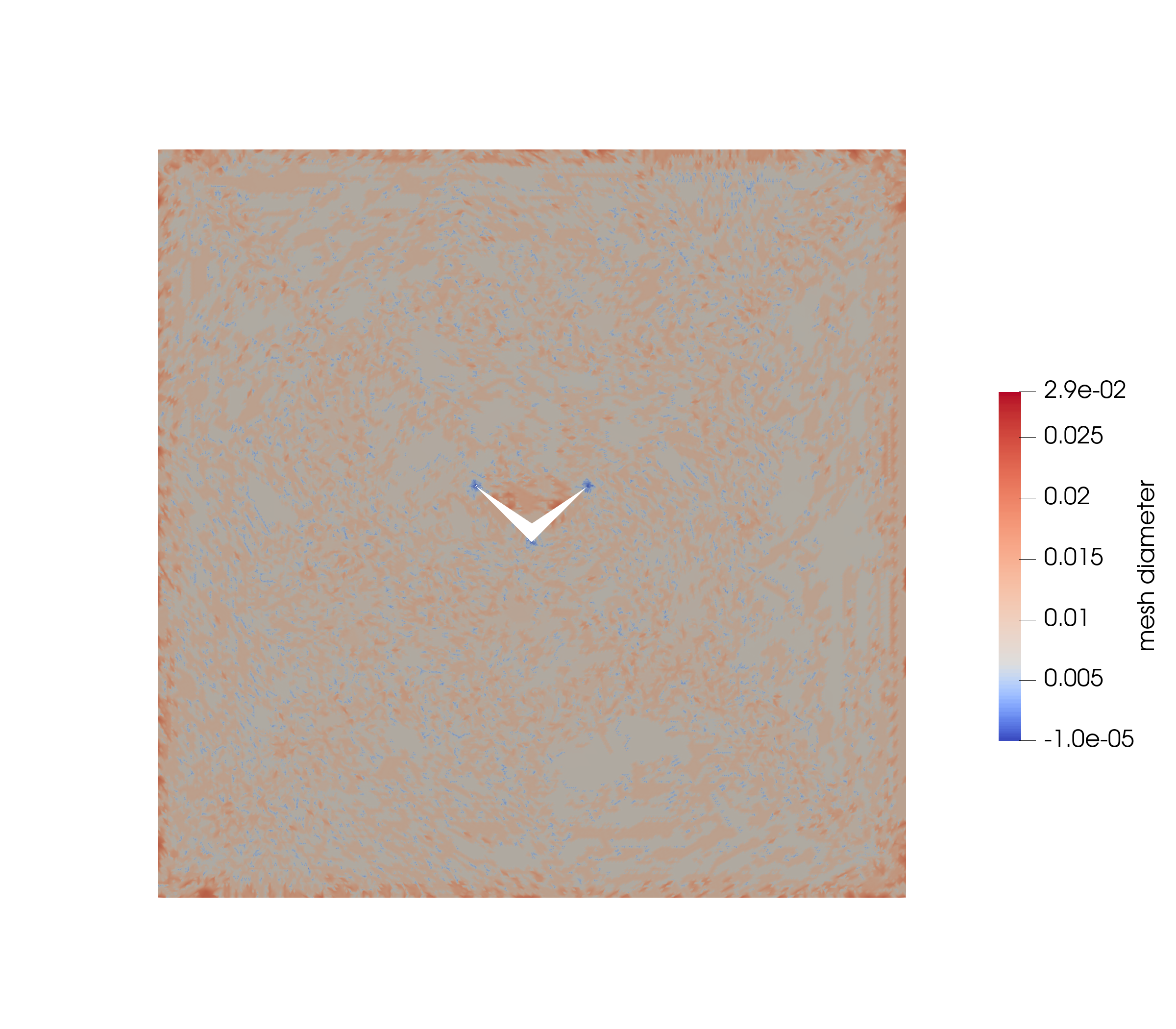}
        \caption{On the left we show the element diameter of the mesh obtained by adaptive refinement for the $\mathcal{P}^1$-conforming finite element method, while on the right we show the same quantity for the mesh obtained by adaptive refinement for the Crouzeix-Raviart finite element method.}
        \label{fig:plane_adpt_mesh}
    \end{figure}
\end{ex}

\bibliographystyle{elsarticle-num}
\bibliography{main.bib}

\end{document}